\documentclass[12pt]{amsart}
\usepackage{fullpage}
\usepackage{latexsym}
\usepackage{a4}
\usepackage{amssymb}
\usepackage{amsbsy}
\usepackage{bm}
\usepackage{algpseudocode}
\usepackage{mathtools}
\makeatletter
\@namedef{subjclassname@2020}{\textup{2020} Mathematics Subject Classification}
\makeatother

\usepackage{enumitem}
\newtheorem{thm}{Theorem}[section]

\newtheorem{lem}[thm]{Lemma}

\newtheorem{cor}[thm]{Corollary}

\theoremstyle{definition}

\numberwithin{equation}{section}

\newcommand{\ob}[1]{{\mathbb{#1}}}

\newcommand{\N}{\Bbb{ N}}
\newcommand{\Z}{\Bbb{ Z}}
\newcommand{\Q}{\Bbb{ Q}}

\newcommand{\setsuchthat}{\mid}

\newcommand{\ul}[1]{\underline{#1}}

\newcommand{\rem}{\operatorname{rem}}
\newcommand{\Tor}{\operatorname{Tor}}
\newcommand{\ord}{\operatorname{ord}}

\newcommand{\algop}[2]{( {#1}, {#2} )}

\newcommand{\R}{\mathbb{R}}

\DeclareMathAlphabet{\mathbfsl}{OT1}{cmr}{bx}{it}
\newcommand{\tupBold}[1]{\mathbfsl{#1}}
\newcommand{\vb}[1]{\tupBold{#1}}

\renewcommand{\emptyset}{\varnothing}

\newcommand{\ar}[1]{^{(#1)}}

\title{Zero testing and equation solving for sparse polynomials
       on rectangular domains}

\author{Erhard Aichinger}
\email[E. Aichinger, corresponding author]{erhard@algebra.uni-linz.ac.at}
\author{Simon Grünbacher}
\email[S. Gr\"unbacher]{simon.gruenbacher@gmail.com}
\author{Paul Hametner}
\email[P. Hametner]{hametnerpaul05@gmail.com}

\address{Institute for Algebra,
  Johannes Kepler University Linz,
  Altenbergerstra\ss e 69, 4040 Linz,
  Austria}

  \thanks{Supported by the Austrian Science Fund (FWF):P33878}
  

\date{\today}

\begin{document}
\bibliographystyle{amsplain}
\begin{abstract}
  We consider sparse polynomials in $N$ variables over a finite field,
  and ask whether they vanish on a set $S^N$, where $S$ is
  a set of nonzero elements of the field.
  We see that if for a polynomial $f$, there is $\vb{c}\in S^N$
  with $f (\vb{c}) \neq 0$, then there is such
  a $\vb{c}$ in every sphere inside
  $S^N$, where the radius of the sphere is bounded by
  a multiple of the logarithm of the number
  of monomials that appear in $f$.
  A similar result holds  for the solutions of the
  equations  $f_1 = \cdots = f_r  = 0$ inside $S^N$.
\end{abstract}

\maketitle

\section{Introduction}
We are guided by the following problem:
Given a finite field $\ob{F}$ with $q$ elements, a polynomial
$f \in \ob{F} [X_1, \ldots, X_N]$ and $S \subseteq \ob{F} \setminus \{0\}$,
determine whether $f (\vb{s}) = 0$ for all  $\vb{s} \in S^N$.
Obviously, this question can be solved by testing $f$ on all $|S|^N$
points in $S^N$. Another method would be to determine
the remainder of $f$ modulo the vanishing ideal
$\langle \prod_{s \in S} (X_i - s) \mid i \in \ul{N} \rangle$ of
$S^N$. Then $f$ vanishes on $S^N$ if and only if this remainder
is $0$. However, in computing the remainder of $f$, we
may obtain intermediate polynomials that contain more monomials
than $f$. Therefore, this approach will only yield a fast algorithm
if $\prod_{s \in S} (X - s)$ contains at most two monomials, which
happens if and only if
$S$ is a coset of a subgroup of $(\ob{F} \setminus\{0\}, \cdot)$.
In this note, we show that if the polynomial is sufficiently sparse,
i.e., contains only few monomials, and has a nonzero
inside $S^N$, then we find such a nonzero
in the proximity of every point in $S^N$. This limits the
number of points at which we evaluate $f$ to a testing
set of size at most $(N \cdot |S|)^{\log_{t} (M)}$, where
$t = \frac{q-1}{q-2}$ and $M$ is the number of monomials in $f$.
As a consequence, we obtain a method to find solutions of sparse systems
of polynomial equations by using that over the finite field $\ob{F}_q$, the
solutions of $f_1 = \cdots = f_r = 0$
are the nonzeros of $\prod_{i=1}^r (1 - f_i^{q-1})$ (Section~\ref{sec:solve}).

The testing set strategy is particularly useful when
we can access the input polynomial $f$ only as a ``black box''
that produces the value $f(\vb{a})$ on input
$\vb{a} \in \ob{F}^N$, we have no access to the
coefficients of $f$, but for some reason, we know that $f$
has at most $M$ monomials. This is a viewpoint
also taken in \cite{CD:OZTA}. There it is
proved that $(N (q-1))^{\log_2 (M)}$ evaluations of $f$
suffice to determine whether
$f$ is $0$ on all of $\ob{F}_q^N$ \cite[p.157]{CD:OZTA} (cf. \cite[p.1061]{GK:FPAF}).

The main content of the present note
lies in the fact that polynomials
that behave similar to a logical conjunction -- we will call such
polynomials \emph{absorbing} -- must often contain many
monomials
(Theorems~\ref{thm:coeffs}, \ref{thm:redcoeffs}, \ref{thm:coeffs2}).
Then, we use a method from \cite{BM:ESAP} (cf. also \cite{Ko:CATE})
to find a nonzero of $f$ in the proximity of
every point of $S^N$ (Theorems~\ref{thm:nonzeroClose},
    \ref{thm:rednonzeroClose}).
    Often, our results hold in more generality for subsets of $\ob{F}^N$ that
    are of the form $\bigtimes_{i=1}^N A_i$, with all $A_i \subseteq \ob{F}$,
    instead of the more special subsets of type $S^N$;
    we will call these subsets \emph{rectangular}.
    
\section{Polynomials with many coefficients} \label{sec:many}

In this section, we prove
that certain polynomials have many
nonvanishing coefficients.

We write $\N$ for the set of positive integers,
and for $N \in \N$, we use $\ul{N}$ to denote the
set $\{1,2,\ldots, N\}$. For a commutative ring  $K$ with~$1$ and $N \in \N$, the set
$Q$ is a \emph{rectangular subset} of $K^N$ if
there are $A_1,\ldots, A_N \subseteq K$ with
$Q = A_1 \times \cdots \times A_N$. Let $E$ be a finite
subset of $\N_0^N$, and let $f = \sum_{\vb{e} \in E} c_{\vb{e}} X_1^{e_1} \cdots X_N^{e_N}$ be a polynomial in $K [X_1, \ldots, X_N]$.
We say that $f$ \emph{contains} a monomial $X_1^{e_1} \cdots X_N^{e_N}$
if $c_{(e_1, \ldots, e_N)} \neq 0$, and we define
\[
M (f) := \# \{ (e_1, \ldots e_N) \in E \mid c_{(e_1, \ldots, e_N)} \neq 0 \}
\]
to be the number of monomials contained in $f$.
Hence for $m \in \N$,
the polynomial $(X_1 + 1)^m \in \Q[X_1, \ldots, X_N]$
satisfies $M((X_1 + 1)^m) = m + 1$,
and over every commutative ring with unit and at least $2$ elements,
$M(\prod_{i=1}^N (X_i + 1)) = 2^N$.
   For a tuple $\vb{a} = (a_1, \ldots, a_N) \in Q$, we say
that a polynomial $f$ is \emph{absorbing at $\vb{a}$ for $Q$} if
     \[
     \text{for all } (x_1, \ldots, x_N) \in Q \,: \, \big(
     (\exists \, i \in \ul{N}  \,:\, x_i = a_i) \Longrightarrow
     f(x_1, \ldots, x_N) = 0 \big).
     \]
     In the case that $K$ is the finite field $\ob{F}_q$ and
     $Q = S^N$, where $S$ is a subgroup of $\algop{\ob{F}_q \setminus\{0\}}{\cdot}$
     of order $d$, then a result by E.\ Kiltz and A.\ Winterhof provides
     a lower bound on the number of monomials in a nonzero absorbing
     polynomial: Lemma~2 of \cite{KW:OTIO} proves that every nonzero
     polynomial $f$ with at least $n$ zeros in $S^N$
     and of degree at most $d-1$ in each variable
     has at least $\frac{d^N}{d^N-n}$ monomials.
     As a corollary we obtain:
     \begin{cor}[cf.\ {\cite[Lemma~2]{KW:OTIO}}]
       Let $N \in \N$, let $\ob{F}$ be a finite field,
       let $S$ be a subgroup of $\algop{\ob{F} \setminus\{0\}}{\cdot}$
       of order $d$, let $f \in \ob{F} [X_1, \ldots, X_N]$, and let
       $\vb{a} \in S^N$. If $f$ is absorbing at
       $\vb{a}$ for $S^N$ and if there is $\vb{b} \in S^N$
       with $f(\vb{b}) \neq 0$, then 
       $M(f) \ge (\frac{d}{d-1})^N$.
     \end{cor}  
     \begin{proof}
       Let
       $f_1$ be the remainder of $f$ modulo
       $(
       X_1^d - 1, \ldots, X_N^d - 1)$.
       Then $M(f) \ge M(f_1)$. Since $f_1$ is absorbing on
       $S^N$, $f_1$ has at most $(d-1)^N$ nonzeros on
       $S^N$. Hence $f_1$ has at least $d^N - (d-1)^N$ zeros
       on $S^N$. Since $f_1 (\vb{b}) \neq 0$,
       \cite[Lemma~2]{KW:OTIO} yields
       $M (f_1) \ge \frac{d^N}{d^N - (d^N - (d-1)^N)} =
       (\frac{d}{d-1})^N$.
   \end{proof}    
     We complement this result by providing other situations in
     which absorbing polynomials will have many monomials.
     To this end, we consider the following class of rings:
     by a \emph{reduced ring}, we understand
     a commutative ring with~$1$ that has no nonzero nilpotent
     elements;  hence a commutative ring with~$1$ is reduced
     if and only if the implication 
     $x^2 = 0 \Rightarrow x = 0$ holds for all ring elements $x$;
     in particular, every field and every integral domain
     is reduced. 
    \begin{thm} \label{thm:coeffs}
  Let $N \in \N$, let $K$ be a reduced ring,
  let $f \in K[X_1, \ldots, X_N]$, and
  let $a_1,b_1, \ldots, a_N, b_N \in K \setminus \{0\}$
  with $a_i \neq b_i$ for all $i \in \ul{N}$.
  We assume that for each $i \in \ul{N}$, there exists $r_i \in \N$
  such that $a_i^{r_i} = b_i^{r_i}$.
  Let $Q := \bigtimes_{i=1}^N \{a_i, b_i\}$. We suppose that
  $f$ is absorbing at $(a_1,\ldots, a_N)$ for $Q$, and that
  \[
  b_1 \cdots b_N \cdot  f(b_1, \ldots, b_N) \neq 0.
  \]
  Then $M(f) \ge \prod_{i=1}^N \frac{r_i}{r_i - 1}$.
\end{thm}

We prepare for the proof with two lemmas.
\begin{lem} \label{lem:1mon}
  Let $N \in \N$, let $K$ be a commutative ring
  with~$1$, 
  let $p \in K [X_1, \ldots, X_N]$, 
  let $a_1,b_1, \ldots, a_N, b_N \in K$
  with $a_i \neq b_i$ for all $i \in \ul{N}$, and let
  $Q := \bigtimes_{i=1}^N \{a_i, b_i\}$. We suppose that
  $p$ is absorbing at $(a_1,\ldots, a_N)$ for $Q$ and that
  $p(b_1, \ldots, b_N) \neq 0$.
  Then $p$ contains a monomial that contains all variables
  $X_1, \ldots, X_N$,
    i.e., a monomial that is divisible by $X_1\cdots X_N$.
\end{lem}     
\begin{proof}
  Let $G := (K, +)$ be the additive group of $K$.
  We let $R$ be the group ring $\Z [G^N]$. Since $G$ is written
  additively, it is convenient to write an element
  $\sum_{\vb{g} \in G^N} z_{\vb{g}} \, \vb{g}$ of this group ring as
  $\sum_{\vb{g} \in G^N} z_{\vb{g}} \, \tau_{\vb{g}}$; the group ring multiplication is
  then determined by $\tau_{\vb{g}} \tau_{\vb{h}} = \tau_{\vb{g}+\vb{h}}$
  for $\vb{g}, \vb{h} \in G^N$. Since $G^N$ is abelian, $R$ is a commutative
  ring. Similar to \cite{AM:CWTR},
  we define a module operation of $R$ on
  its module $K [X_1, \ldots, X_N]$ by
  \[
      \tau_{(c_1, \ldots, c_N)} \,* \, f(X_1, \ldots, X_N) :=
      f (X_1 + c_1, \ldots, X_N + c_N)
  \]
  for $(c_1, \ldots, c_N) \in G^N$ and $f \in K[X_1, \ldots, X_N]$.
   Taking $\vb{u}_i$ to be the $i$-th unit vector $(0,\ldots,0,
  \underbrace{1}_{i\,\text{th place}},0, \ldots 0)$
  in $K^N$,
  $(b_i - a_i ) \vb{u}_i$ denotes the tuple $(0,\ldots,0,
  \underbrace{b_i-a_i}_{i\,\text{th place}},0, \ldots 0) \in G^N$,
  and we use this tuple to define $r \in R$ by
  \[
  r := \prod_{i = 1}^N (\tau_{(b_i - a_i) \vb{u}_i} - 1).
  \]
  
  We first show that for every monomial $m$ that is
  not divisible by $X_1\cdots X_N$, we have
  \begin{equation} \label{eq:r1}
    r * m  = 0.
  \end{equation}
  Since $m$ is not divisible by $X_1 \cdots X_N$, there is  $j \in \ul{N}$ be such that $X_j$ does not appear in $m$.
    Let
    $s := \prod_{i \in \ul{N} \setminus \{j\}} (\tau_{(b_i - a_i) \vb{u}_i} - 1)$
   and $t := \tau_{(b_j-a_j) \vb{u}_j} - 1$.
   Then $st = r$ and 
   we have
      \begin{multline*}
        t * m =
        m \, (X_1, \ldots, X_{j-1}, X_j + (b_j- a_j), X_{j+1}, \ldots, X_N) \\
        -
        m \, (X_1, \ldots, X_{j-1}, X_j, X_{j+1}, \ldots, X_N) 
        =
        0.
      \end{multline*}
    Thus 
    $r * m = (st) * m = s*(t*m) = s*0 = 0$, establishing~\eqref{eq:r1}.
    By expanding $r$, we obtain
    \begin{equation} \label{eq:expanded} 
      (r * p) \, (a_1,\ldots,a_N) = \sum_{I \subseteq \ul{N}}
      (-1)^{N-|I|} p(c_1\ar{I}(a_1, b_1), \ldots,
                    c_N\ar{I}(a_N, b_N)),
     \end{equation}
     where $c_i\ar{I} (a_i, b_i) = b_i$ if $i \in I$ and
     $c_i\ar{I} (a_i, b_i) = a_i$ if $i \not\in I$.
     Since $p$ is absorbing at $(a_1, \ldots, a_N)$ for
     $Q$, we have $p (x_1, \ldots, x_N) = 0$
     for $\vb{x} \in Q \setminus \{(b_1, \ldots, b_N)\}$.
     Therefore,
     all the $2^N - 1$  summands in the right hand side of~\eqref{eq:expanded}
     with $I \neq \ul{N}$ vanish, and thus
     we have
     $(r * p) (a_1, \ldots, a_N)  = (-1)^0 \, p(b_1, \ldots, b_N) \neq 0$.
     Thus $r*p \neq 0$.

    Hence $p$ contains a monomial $m$ with $r*m \neq 0$.
    Then this monomial $m$ contains all variables.
\end{proof}
  
\begin{lem} \label{lem:q}
  Let $N \in \N$, 
  let $A_1, \ldots, A_N$ be finite sets with $|A_i| > 1$ for
  all $i \in \ul{N}$, let
  $Q := A_1 \times \cdots \times A_N$, and let 
  $S$ be a subset of $Q$.
  For $i \in \ul{N}$, let $r_i := |A_i|$, and suppose
  that $|S| < \prod_{i=1}^N \frac{r_i}{r_i - 1}$.
  Then there is a tuple $(q_1,\ldots, q_N) \in Q$ such that
  for every $(s_1, \ldots, s_N) \in S$, there is an
  $i \in \ul{N}$ such that $q_i = s_i$.
\end{lem}
\begin{proof}
  The idea of the proof is that each element of $S$ excludes
  only $\prod_{i=1}^N (r_i - 1)$ possible choices for $\vb{q}$,
  and the bound on $|S|$ guarantees that not all elements of $Q$
  become excluded. Formally, we proceed as follows:
  for each $\vb{s} \in S$, let
  \[
  D (\vb{s}) := \{ \vb{q} \in Q \mid \text{ for all } i \in \ul{N}:
  q_i \neq s_i \}.
  \]
  Then $|D (\vb{s})| = \prod_{i=1}^N (r_i - 1)$.
  Hence
  \[
     |\bigcup_{\vb{s} \in S} D (\vb{s})| \le
     |S| \cdot \prod_{i=1}^N (r_i - 1) <
     (\prod_{i=1}^N \frac{r_i}{r_i - 1}) \cdot
     (\prod_{i=1}^N (r_i - 1)) = |Q|.
   \]
   Therefore there is an element $\vb{q} \in Q$ such that
   for all $\vb{s} \in S$, we have $\vb{q} \not\in D(\vb{s})$.
   Then for every $\vb{s} \in S$, since $\vb{q} \not\in D(\vb{s})$,
   there is a $j \in \ul{N}$
   such that $q_j = s_j$.
\end{proof}

\begin{proof}[Proof of Theorem~\ref{thm:coeffs}]
  Seeking a contradiction, we suppose that
  \[
  f = \sum_{\vb{e}\in E} c_{\vb{e}} X_1^{e_1} \cdots X_N^{e_N}
  \]
  with 
  $|E| < \prod_{i=1}^N \frac{r_i}{r_i - 1}$.
  Denoting the remainder of the division of $e$ by $r$ in $\Z$
  by $\rem (e,r)$,
  we define \[ S := \{ (\rem (e_1, r_1), \ldots, \rem (e_N, r_N)) \mid
  (e_1, \ldots, e_N) \in E \} \] and
  $Q_1 := \bigtimes_{i=1}^N \{0,1, \ldots, r_i - 1\}$.
  Then 
  $S \subseteq Q_1$ and 
  $|S| \le |E|$.
  By Lemma~\ref{lem:q}, there exists $(q_1, \ldots, q_N) \in Q_1$
  such that for all $(e_1, \ldots, e_N) \in E$, there
  is $i \in \ul{N}$ such that $\rem (e_i, r_i) = q_i$.
  Let
  \[
  p(X_1,\ldots,X_N) := X_1^{r_1 - q_1} \cdots X_N^{r_N - q_N} f(X_1, \ldots, X_N).
  \]
  Let $m = c_{\vb{e}} X_1^{s_1} \cdots X^{s_N}$ with
  $s_i = r_i -q_i + e_i$ be a monomial from $p$.
  By the choice of $\vb{q}$, there is
  $j \in \ul{N}$ be such that $\rem (e_j, r_j) = q_j$.
   Then $s_j = 
      r_j - q_j + e_j = r_j - \rem(e_j, r_j) + e_j$, and therefore
      $r_j$ divides $s_j$, and thus $a_j^{s_j} = b_j^{s_j}$.
      We set
      \[
      m' := m(X_1, \ldots, X_{j-1}, a_j, X_{j+1}, \ldots, X_N).
      \]
      Then since $a_j^{s_j} = b_j^{s_j}$, we have
      $m(\vb{x}) = m' (\vb{x})$ for all $\vb{x} \in Q$.
      Hence $p_1 = p - m+ m'$ induces the same function on $Q$
      as $p$ and satisfies $M(p_1) \le M(p)$.
      Repeating this for all monomials in $p$, we
      obtain a polynomial $p_2$ that induces the same function on
      $Q$ as $p$ and has no monomial divisible by $X_1 \cdots X_N$.
      Furthermore,
      $M(p_2) \le M(p) = M(f) \le |E| < \prod_{i=1}^N  \frac{r_i}{r_i - 1}$.
      Now
      $p_2 (\vb{x}) = 0$ for $\vb{x} \in Q \setminus \{(b_1, \ldots, b_N)\}$
      and $p_2 (b_1\ldots,b_n) =
      b_1^{r_1 - q_1} \cdots b_N^{r_N - q_N} \cdot f(b_1, \ldots, b_N)$.
      Suppose that
      \[
      b_1^{r_1 - q_1} \cdots b_N^{r_N - q_N} \cdot f(b_1, \ldots, b_N)
      =
      0.
      \]
      Then setting $k := \max \,\{ r_i - q_i \mid i \in \ul{N} \}$,
      we obtain
      \[
      (b_1 \cdots b_N \cdot f(b_1, \ldots, b_N))^k = 0,
      \]
      which implies $b_1 \cdots b_N \cdot f(b_1, \ldots, b_N) = 0$
      because $K$ is reduced.
      This contradicts the assumptions, and therefore
      $p_2 (b_1\ldots,b_n) \neq 0$.
      This polynomial $p_2$ contradicts Lemma~\ref{lem:1mon}.

      Hence $f$ contains at least $\prod_{i=1}^N \frac{r_i}{r_i - 1}$
      monomials.
\end{proof}

\section{Lower bounds on the density of nonzeros}
  For a commutative ring $K$ with~$1$, let 
  $\Tor (K)$ be the torsion part of the (multiplicative) group
  of units of $K$, i.e., the set
  $\{t \in K \mid \exists n \in \N \, : \, t^n = 1\}$; hence
  if $K$ is a field, $\Tor (K)$ is the set of roots of unity of $K$.
  For $t \in \Tor (K)$, we write
  $\ord (t)$ for its multiplicative order, i.e., for
  $\min \{n \in \N \mid t^n = 1\}$.
  For two elements $\vb{a}, \vb{b} \in K^N$, their
  \emph{Hamming-distance} $d_H(\vb{a}, \vb{b})$ is defined
  by \(
     d_H (\vb{a}, \vb{b}) := \# \{ i \in \ul{N} \mid
     a_i \neq b_i \}.
    \) 
  \begin{thm} \label{thm:nonzeroClose}
    Let $N \in \N$,
    let $K$ be an integral domain,
    let $A_1, \ldots, A_N$ be
    finite nonempty subsets of $\Tor (K)$,
    and let $Q := \bigtimes_{i=1}^N A_i$.
    Let
    \(
    r := \max \, (\{\ord(u/v) \mid u, v \in A_i, \, i \in \ul{N} \} \cup \{2\}),
    \)
    let
    $t := \frac{r}{r-1}$,
    and let $p \in K [X_1, \ldots, X_N]$.
        We assume that there is $\vb{c} \in Q$ such that
    $p (\vb{c}) \neq 0$.
    Then for every $\vb{a} \in Q$, there is $\vb{b} \in Q$
    with $p(\vb{b}) \neq 0$ and
    \(
    d_H (\vb{a}, \vb{b}) \le \log_t (M (p)).
    \)
  \end{thm}
  \begin{proof}
    We note that the main idea of this proof is taken from the proof
    of \cite[Theorem~2]{BM:ESAP}.
      Let $\vb{b} \in Q$ be such that $p (\vb{b}) \neq 0$ and
   $d_H (\vb{a}, \vb{b})$ is minimal in
    $\{ d_H (\vb{a}, \vb{z}) \mid \vb{z} \in Q, \, p (\vb{z}) \neq 0 \}$.
      Let $k := d_H (\vb{a}, \vb{b})$. In the case $k = 0$,
      the claimed inequality clearly holds. For the case $k > 0$,
      we let $i_1, \ldots, i_k \in \ul{N}$
    with $i_1 < i_2 < \ldots < i_k$ and
    $a_{i_1} \neq b_{i_1}, \ldots, a_{i_k} \neq b_{i_k}$.
    Let
    \begin{multline*} \label{eq:fp}
     f (X_1,\ldots, X_k) :=
     p (a_1, \ldots, a_{i_1 - 1}, X_1, a_{i_1 + 1}, \ldots,
           a_{i_2 - 1}, X_2, a_{i_2 + 1}, \ldots, \\
       \ldots, a_{i_k - 1}, X_k, a_{i_k + 1}, \ldots, a_N).
    \end{multline*}
    Let $Q_1 := \bigtimes_{m=1}^k \{a_{i_m}, b_{i_m} \}$.
    By the minimality of $d_H (\vb{a}, \vb{b})$,
        $f (\vb{x}) = 0$ for all
    $\vb{x} \in Q_1 \setminus \{ (b_{i_1}, \ldots, b_{i_k}) \}$, and
      $f (b_{i_1}, \ldots, b_{i_k}) = p (b_1, \ldots, b_N) \neq 0$.
        Now by Theorem~\ref{thm:coeffs}, we obtain
    \[
        M (p) \ge \prod_{m = 1}^k \frac{\ord(b_{i_m}/a_{i_m})}{\ord (b_{i_m}/a_{i_m}) - 1}
        \ge \prod_{m = 1}^k \frac{r}{r - 1}
        = \left(\frac{r}{r-1}\right)^k = t^k,
    \]
    and therefore $\log_t (M(p)) \ge k$.  Hence
    $d_H (\vb{a}, \vb{b}) \le \log_t (M (p))$.
  \end{proof}
  From this result, we obtain a lower bound for the number
  of nonzeros of a polynomial. In order to express this
  bound, we use two functions from coding theory (cf.\ \cite{vL:ITCT}):
  $\operatorname{Vol}_q (n, k) := \sum_{i=0}^{\lfloor k \rfloor} {n \choose i} (q-1)^i$, which is the number of points in $\{0,\ldots, q-1\}^n$ with at most
  $k$ entries different
  from~$0$, and 
  the entropy function 
  $H_q : [0,1] \to \R$, 
  $H_q(0) = H_q(1) = 0$ and $H_q(x) = x \log_q (q-1) - x \log_q (x) + (1-x) \log_q (1-x)$ for $x$ with $0 < x < 1$, which is the entropy of a source
  producing $q$ symbols with probabilities
  $(1-x, \frac{x}{q-1}, \ldots,
  \frac{x}{q-1})$ divided by $\log_2 (q)$. 
  \begin{cor} \label{cor:numbernonzeros}
    Let $N \in \N$,
    let $K$ be an integral domain, let $S$ be
    a finite subset of $\Tor (K)$ with $|S| \ge 2$,
    let $Q := S^N$, and let $s := |S|$.
    Let
    \(
    r := \max \{\ord(u/v) \mid u,v \in S \},
    \)
    let
    $t := \frac{r}{r-1}$, 
    and let $p \in K[X_1, \ldots, X_N]$.
        Let $W := \{ \vb{b} \in S^N \mid p (\vb{b}) \neq 0 \}$.
    We assume $W \neq \emptyset$. Then we have
   \[
         |W|
         \ge
         \frac{s^N}{\operatorname{Vol}_s (N,  \log_t (M(p)))}
         \ge \frac{s^N}{(N s)^{\log_t (M(p))}},
   \]      
   and if $M(p) \le t^{\frac{N(q-1)}{q}}$,
   $|W| \ge s^{N (1 - H_q (\frac{\log_t (M(p))}{N}))}$.
  \end{cor}
  \begin{proof}
  For
  $\vb{b} \in Q$ and $R>0$, let $B (\vb{b}, R) := \{ \vb{a} \in Q \mid d_H (\vb{a},\vb{b})
  \le R\}$. From Theorem~\ref{thm:nonzeroClose}, we obtain
  \[
  \bigcup_{\vb{b} \in W} B (\vb{b}, \log_t (M(p))) = Q,
  \]
  which implies
  $|W| \cdot \operatorname{Vol}_s (N, \log_t (M(p))) \ge s^N$.
   The other equalities follow from known upper bounds for
  $\operatorname{Vol}_s (N, k)$ for $k \in \N_0$: for the comparison with
  the entropy function $H_q$, see \cite[p.65]{vL:ITCT}.
  The bound
  $|\operatorname{Vol}_s (N,k)| \le
   (Ns)^k$
  is of elementary nature:
  we map every $\vb{a}$ with $d_H (\vb{b}, \vb{a}) \le k$
  to a list $(i_1, s_1), \ldots, (i_k, s_k)$ of elements
  of $\ul{N} \times S$ with the property that
  $a_{i_l} = s_l$ for $l \in \ul{k}$ and $a_j = b_j$ for
  $j \not\in \{i_l \mid l \in \ul{k} \}$. In other words,
  for every $\vb{a}$ we list the changes made when turning
  $\vb{b}$ into $\vb{a}$. Since there are at most
  $(N |S|)^k$ such lists, we obtain  $|\operatorname{Vol}_k (N,k)| \le
   (Ns)^k$.
  \end{proof}
  We note that in the case that $K = \ob{F}_q$ and $S$ is
  a subgroup of $\algop{\ob{F}_q\setminus\{0\}}{\cdot}$, \cite{KW:OTIO}
  provides the much stronger bound $|W| \ge \frac{s^N}{M(p)}$.
  Our approach via absorbing functions allows to generalize
  the case $s=2$ of this result:
       \begin{thm} \label{thm:2elements}
  Let $N \in \N$, let $K$ be an integral domain,
  let $a_1,b_1, \ldots, a_N, b_N \in K \setminus \{0\}$
  with $a_i \neq b_i$ for all $i \in \ul{N}$.
  We assume that there is $r \in \N$ such that for each $i \in \ul{N}$,
  we have $a_i^{r} = b_i^{r}$. Let $t := \frac{r}{r-1}$,
  let $Q := \bigtimes_{i=1}^N \{a_i, b_i\}$, 
  let $p \in K[X_1, \ldots, X_N]$, and
   let $W := \{ \vb{c} \in Q \mid p (\vb{c}) \neq 0 \}$.
   We assume $W \neq \emptyset$. Then we have
   \(
         |W|
         \ge
         2^{N - \log_t (M (p))}.
   \)      
 \end{thm}
 \begin{proof}
   We proceed by induction on $|W|$.
   We first consider the case $|W| = 1$.
   In this case there is exactly one $\vb{u} \in Q$ with
   $p(\vb{u}) \neq 0$. Let $\vb{v}$ be the vertex opposite
   to $\vb{u}$ in $Q$, i.e., the unique $\vb{v} \in Q$ with
   $d_H (\vb{u}, \vb{v}) = N$. Then $p$ is absorbing at $\vb{v}$ for
   $Q$. By Theorem~\ref{thm:coeffs}, we have $M(p) \ge t^N$, and thus
   $2^{N - \log_t (M (p))} \le 1 \le |W|$.
   For the induction step, suppose that $|W| \ge 2$. Since
   $W \neq \emptyset$, we have $p \neq 0$ and thus $M(p) \ge 1$.
   In the case $N=1$, we have $2^{N - \log_t (M(p))} \le 2^N = 2 \le |W|$.
   If $N \ge 2$, we choose $\vb{u}, \vb{v} \in W$ with $\vb{u} \neq \vb{v}$.
   Let $i \in \ul{N}$ be such that $u_i \neq v_i$; without loss of generality,
   $u_i = a_i$ and $v_i = b_i$.
   We can then write
   \begin{equation} \label{eq:W}
   W = \{ \vb{c} \in W \setsuchthat c_i = a_i \} \cup
   \{ \vb{c} \in W \setsuchthat c_i = b_i \}.
   \end{equation}
   as the union of two disjoint nonempty sets.
   We define
   \begin{equation*}
      \begin{array}{rcl}
        p_1 (X_1, \ldots, X_{N-1}) & := &
        p(X_1, \ldots, X_{i-1}, a_i, X_i, \ldots, X_{N-1}), \\
        p_2 (X_1, \ldots, X_{N-1}) & := &
        p(X_1, \ldots, X_{i-1}, b_i, X_i, \ldots, X_{N-1}).
      \end{array}
   \end{equation*}
   For $j \in \{1,2\}$, let
   $W_j := \{\vb{c}' \in K^{N-1} \mid  p_j (\vb{c}') \neq 0\}$.
   Then $|\{ \vb{c} \in W \,:\, c_i = a_i \}| = |W_1|$.
   Since $W_1$ is not
   empty, the induction hypothesis yields $|W_1| \ge 2^{N-1 - \log_t (M(p_1))}$.
   Similarly, $|\{ \vb{c} \in W \,:\, c_i = b_i \}| = |W_2|$
   and $|W_2| \ge 2^{N-1 - \log_t (M(p_2))}$.
   Hence from~\eqref{eq:W}, we obtain
   $|W| \ge 2^{N-1 - \log_t (M(p_1))} + 2^{N-1 - \log_t(M(p_2))} \ge
            2^{N-1 - \log_t (M(p))} + 2^{N-1 - \log_t (M(p))} =
            2 \cdot  2^{N-1 - \log_t (M(p))} = 2^{N - \log_t (M(p))}$,
   which completes the induction step.
 \end{proof}
 
  \section{Applications to zero testing} \label{sec:zt}
 Theorem~\ref{thm:nonzeroClose} provides a ``black box'' test
 to check whether a polynomial over a finite field
 $\ob{F}$ vanishes whenever all its arguments
 are chosen from a given subset~$S$.
  \begin{cor} \label{cor:testset}
    Let $\ob{F}$ be a finite field with $q > 2$ Elements, let
    $t := \frac{q-1}{q-2}$, and
    let $S \subseteq \ob{F} \setminus \{0\}$. There is an algorithm
    that given $M, N \in \N$ and a polynomial
    $p \in \ob{F}[X_1, \ldots, X_N]$ that is the sum of at most $M$ monomials,
    decides whether $p(\vb{s}) = 0$ for all $\vb{s} \in S^N$
    and uses at most
    \[
    \max(1, \left( \begin{smallmatrix} N \\ \lfloor \log_t(M) \rfloor
                   \end{smallmatrix} \right) )\cdot  |S|^{
                 \lfloor \log_t (M) \rfloor } 
    \]
    evaluations of $p$.
  \end{cor}
  \begin{proof}
    The algorithm proceeds as follows:
    we choose an $s \in S$. Then
    we check whether $p (\vb{u}) = 0$ for all
    $(u_1, \ldots, u_N) \in S^N$ that contain at most
    $k := \lfloor \log_t(M) \rfloor$ entries 
    that are different from $s$. Since the order of every
    element in $\ob{F}$ is at most $q-1$,
    Theorem~\ref{thm:nonzeroClose} yields that if $p$ vanishes
    on all these arguments, then $p$ vanishes on all of $S^N$.
    Now if $k \le N$, then every tuple in $S^N$ with at most
    $k$ entries that are different from $s$ can be chosen
    by first choosing $T$ as one of the ${N \choose k}$ subsets
    of $\ul{N}$ with $k$ elements, then setting the $u_i$ with
    $i \in T$ to any value in $S$ (there are $|S|^k$ possibilities
    for this) and finally setting all $u_i$ with $i \not\in T$ to $s$.
    Hence there are at most ${N \choose k} |S|^k$ such tuples.
    If $k > N$, then $p$ is tested on all tuples in $S^N$,
    and in this case we also have
    $|S^N| \le  \max(1, {N \choose k}) \cdot  |S|^{k}$.
 \end{proof}   
  Hence if we measure the size of the input polynomial in such a way
  that the size $n$ of $p$ is at least $\max(N,M)$, where $N$ is its
  number of variables and $M$ is its number of monomials,
  we obtain an algorithm of time complexity in $O(n^{c \log(n)})$ (with
  $c > 0$)  that determines whether $p$ vanishes whenever all its arguments
  are chosen from some given $S \subseteq \ob{F} \setminus \{0\}$.

  We note that Corollary~\ref{cor:testset} does not require
  any degree bounds on $p$. Degree bounds could significantly
  restrict the scope of the theorem.
  For example, the polynomial
  $p = \prod_{i=1}^N X_i^2$ over $\ob{F}_4 = \{0,1,\omega,\omega^2\}$ contains only one monomial,
  but for $S = \{\omega, \omega^2\}$, the only $p'$ that agrees
  with $p$ on $S^N$ and satisfies $\deg_{X_i} (p') < |S|$ for all $i \in \ul{N}$
  is $p' = \prod_{i = 1}^N (X_i + 1)$, which contains $2^N$ monomials.

  \section{Polynomials over fields with bounded degree}
  For polynomials over fields, we can sometimes improve the
  bound given in Theorem~\ref{thm:coeffs}. The next theorem, however,
  requires bounds on the degree of $f$.
  \begin{thm} \label{thm:redcoeffs}
    Let $N \in \N$, let $\ob{K}$ be a field, let
    $A_1, \ldots, A_N$ be finite subsets of $\ob{K}$,
    and let
    $\vb{a} = (a_1, \ldots, a_N) \in \bigtimes_{i=1}^N (A_i \setminus \{0\})$.
    Let $f \in \ob{K}[X_1, \ldots, X_N]$ be a polynomial
    such that $\deg_{X_i} (f) < |A_i|$ for all $i \in \ul{N}$, and
    let $Q := \bigtimes_{i=1}^N A_i$. We assume that
    $f$ is absorbing at $\vb{a}$ for $Q$, and that 
    there is $\vb{b} \in Q$ with $f (\vb{b}) \neq 0$.
    Then $M(f) \ge 2^N$.
  \end{thm}
   The proof uses the following lemma.
    \begin{lem} \label{lem:comb}
    Let $N \in \N$,  $A_1, \ldots, A_N$ be nonempty sets,
    and let
    $E \subseteq A_1 \times \cdots \times A_N$ be such that
    \begin{multline} \label{eq:second}
    \text{for all }(e_1, \ldots, e_N) \in E \text{ and for all }
    j \in \ul{N}, \\
     \text{ there is } e_j' \in A_j \text{  with }
      e_j' \neq e_j \text{ and }
      (e_1, \ldots, e_{j-1}, e_j', e_{j+1}, \ldots, e_N) \in E.
    \end{multline}  
    Then $|E| \ge 2^N$.
  \end{lem}
  \begin{proof}
    We proceed by induction on $N$.
    In the case $N = 1$, the assumptions guarantee that
    $A_1$ is nonempty and contains at least $2$ elements.

    For the induction step, we let $N \ge 2$ and
    assume that the result holds
    for all collections of $N-1$ sets. We assume that
    $E \subseteq A_1 \times \cdots \times  A_N$ satisfies~\eqref{eq:second}.
    We define
    \[
    E' := \{ (e_1, \ldots, e_{N-1}) \mid
    \exists e \in A_N \,:\, (e_1, \ldots, e_{N-1}, e) \in E \}
    \]
    as the projection of $E$ to its first $N-1$ components.
    Next we show that $E'$ satisfies the assumptions made
    in~\eqref{eq:second}. For this purpose, we pick
    $(e_1, \ldots, e_{N-1}) \in E'$ and $j \in \ul{N-1}$.
    There is $e \in A_N$ with
    $(e_1, \ldots, e_{N-1}, e) \in E$. Therefore, since
    $E$ satisfies~\eqref{eq:second}, there is $e_j' \in A_j$ with
    $(e_1, \ldots, e_{j-1}, e_j', e_{j+1}, \ldots, e_{N-1}, e) \in E$
    and $e_j' \neq e_j$. Hence
    $(e_1, \ldots, e_{j-1}, e_j', e_{j+1}, \ldots, e_{N-1}) \in E'$,
    which completes the proof that $E'$  satisfies the assumptions
    made in~\eqref{eq:second}.
    Therefore by the induction hypothesis,
    we have $|E'| \ge 2^{N-1}$.
    Now consider the projection $\pi:E \to E', (e_1, \ldots, e_N)
    \mapsto (e_1, \ldots, e_{N-1})$. Then the assumptions on $E$
    guarantee that for every $(e_1, \ldots, e_{N-1})$
    the pre-image $\pi^{-1} (\{ (e_1, \ldots, e_{N-1}) \} )$
    contains at least two different elements from $E$.
    Hence $|E| \ge 2^{N-1} \cdot 2 = 2^N$, completing the induction proof.
 \end{proof}

\begin{proof}[Proof of Theorem~\ref{thm:redcoeffs}]
    For $i \in \ul{N}$, let $r_i := |A_i|$,
    and let $D := \bigtimes_{i=1}^N \{0,1, \ldots, r_i - 1\}$.
    Suppose that 
        \[
        f = \sum_{\vb{d} \in D} c(d_1, \ldots, d_N) X_1^{d_1} \cdots X_N^{d_N}.
        \]
        We have to show that $E := c^{-1} (\ob{K} \setminus \{0\})$
        has at least $2^N$ elements. Intending to use Lemma~\ref{lem:comb},
        we let $(e_1, \ldots, e_N) \in E$ and 
        $j \in \ul{N}$. We know that
        \[
           g(Y_1, \ldots, Y_{j-1}, Y_{j+1}, \ldots, Y_{N}) := 
           f(Y_1,\ldots, Y_{j-1}, a_j, Y_{j+1}, \ldots, Y_{N})
        \]   
        is a polynomial in $\ob{K} [
          Y_1, \ldots, Y_{j-1}, Y_{j+1}, \ldots, Y_{N}
                                   ]$  
         that
        is $0$ on $Q' := \prod_{i \in \ul{N} \setminus \{j\}} A_i$.
        Hence $g$ lies in the vanishing ideal of $Q'$, which
        is generated by
        $B = \{ \prod_{\alpha \in A_i} (Y_i - \alpha) \setsuchthat
        i \in \ul{N} \setminus \{j\}  \}$ (cf. \cite[Theorem~1.1]{Al:CN}). Since the leading monomials of
        the polynomials in $B$ are coprime, $B$ is a Gr\"obner basis
        of $Q'$ (with respect to every variable ordering, cf.
        \cite[p.89, Exercise 11]{CLO:IVAA4}),
        and by the degree bounds on $f$, we obtain that
        $g$ is
        in reduced form with respect to $B$, and therefore
        (by \cite[Theorem~1.6.2]{AL:AITG} or \cite[p.84, Corollary~2]{CLO:IVAA4})
        $g = 0$. Hence the coefficient of  $Y_1^{e_1} \cdots Y_{j-1}^{e_{j-1}}
        Y_{j+1}^{e_{j+1}} \cdots Y_{N}^{e_N}$ in $g$ is $0$, which
        implies
            \begin{equation} \label{eq:sum0}
         \sum_{k = 0}^{\deg_{X_j} (f)} c(e_1, \ldots, e_{j-1}, k, e_{j+1},
                                   \ldots, e_N) \, a_j^k = 0.
        \end{equation}
        Since $(e_1,\ldots, e_N) \in E$ and $a_j \neq 0$,
        we have $c(e_1, \ldots, e_{j-1}, e_j, e_{j+1}, \ldots, e_N)\, a_j^{e_j} \neq 0$.    Thus by~\eqref{eq:sum0},
        there is $e_j' \in \N_0$ with
        $e_j' \le \deg_{X_j} (f)$
        such that $e_j' \neq e_j$
        and
        $c(e_1, \ldots, e_{j-1}, e_j', e_{j+1}, \ldots, e_N)\, a_j^{e_j'} \neq 0$.
        Then $(e_1, \ldots, e_{j-1}, e_j', e_{j+1}, \ldots, e_N)$
        is an element of $E$.
        Now Lemma~\ref{lem:comb} yields $|E| \ge 2^N$.
  \end{proof}
As a consequence, the spheres containing nonzeros can sometimes be
chosen smaller than in Theorem~\ref{thm:nonzeroClose}.

    \begin{thm} \label{thm:rednonzeroClose}
    Let $N \in \N$,
    let $\ob{K}$ be a field, let $A_1, \ldots, A_N$ be
    finite subsets of $\ob{K}$,
    let $Q := \bigtimes_{i=1}^N A_i$,
    and let $p \in \ob{K}[X_1, \ldots, X_N]$ 
    be a polynomial with
    $\deg_{X_i} (p) < |A_i|$ for all $i \in \ul{N}$.
    We assume that there is $\vb{c} \in Q$ such that
    $p (\vb{c}) \neq 0$.
    Then for every
    $\vb{a} \in \bigtimes_{i=1}^N (A_i \setminus \{0\})$,
    there is $\vb{b} \in Q$
    with $p(\vb{b}) \neq 0$ and
    \(
    d_H (\vb{a}, \vb{b}) \le \log_2 (M(p)).
    \)
   \end{thm}
 \begin{proof}
   We proceed as in the proof of Theorem~\ref{thm:nonzeroClose}
   and observe that the polynomial
    \begin{multline*} \label{eq:fp}
     f (X_1,\ldots, X_k) :=
     p (a_1, \ldots, a_{i_1 - 1}, X_1, a_{i_1 + 1}, \ldots,
           a_{i_2 - 1}, X_2, a_{i_2 + 1}, \ldots, \\
       \ldots, a_{i_k - 1}, X_k, a_{i_k + 1}, \ldots, a_N)
    \end{multline*}
    is absorbing for $Q' := \bigtimes_{l=1}^k A_{i_l}$
    at $(a_{i_1}, \ldots, a_{i_k})$.
    Then Theorem~\ref{thm:redcoeffs} yields
    that $f$ has at least $2^k$ monomials. Therefore,
    also $p$ has at least $2^k$ monomials, which
    means $M(p) \ge 2^k$, and therefore
    $d_H (\vb{a}, \vb{b}) = k \le \log_2 (M(p))$. 
 \end{proof}

  Theorem~\ref{thm:rednonzeroClose}
 allows to  formulate \cite[Corollary~2.6]{CD:OZTA}
    in a more general setting.
 \begin{cor}
   Let $M, N \in \N$, let $\ob{K}$ be a field, let $a \in \ob{K} \setminus \{0\}$, let
   $f \in \ob{K}[X_1, \ldots, X_N]$ be a polynomial with
   at most $M$ monomials that satisfies
   $\deg_{X_i} (f) < 2$ for all
   $i \in \ul{N}$, and let $Q := \{a, 0\}^N$.
   If $f$ vanishes on all $\vb{z} \in Q$ with 
   at most $\log_2 (M)$ zero entries, then $f$ vanishes on
   all of $Q$.
 \end{cor}
 \begin{proof}
   Suppose that $f$ does not vanish on all of $Q$. Then
   by Theorem~\ref{thm:rednonzeroClose}, there is a $\vb{z} \in Q$ with $f(\vb{z}) \neq 0$ and
   $d_H (\vb{z}, (a,a,\ldots, a)) \le \log_2 (M)$.
   Hence $\vb{z}$ has at most $\log_2 (M)$ zero entries.
 \end{proof}  
   
 \section{Polynomials over fields on rectangular domains containing
   $0$}
 In this Section, we consider another case in which one can
 guarantee the existence of many monomials;
 in this case,
 the domain of the considered
 polynomial
 function is of the form $S = \{a, 0\}^N$ or $S = \bigtimes_{i=1}^N \{a_i,0\}$,
 where $a, a_1, \ldots, a_N$ are nonzero field
 elements, and in contrast to Theorem~\ref{thm:redcoeffs}, no
 degree bounds are required.
\begin{thm} \label{thm:coeffs2}
  Let $N \in \N$, let $\ob{K}$ be a field,
  let $f \in \ob{K} [X_1, \ldots, X_N]$, and
  let $a_1, \ldots, a_N \in \ob{K} \setminus \{0\}$.
  Let $Q := \bigtimes_{i=1}^N \{a_i, 0\}$. We suppose that
  $f$ is absorbing at $(a_1,\ldots, a_N)$ for $Q$, and that
  $f(0,\ldots, 0) \neq 0$. 
  Then $M(f) \ge 2^N$.
\end{thm}
\begin{proof}
  Let $I(Q)$ be the ideal of $\ob{K}[X_1,\ldots,X_N]$ that
  consists of all polynomials that vanish on all of $Q$.
  By \cite{Al:CN}, this ideal is generated by
  $G = \{X_i^2 - a_i X_i \mid i \in \ul{N}\}$. 
  By the assumptions, $f$ is congruent to
  \[
  h = \frac{f(0,\ldots,0)}{(-1)^N \prod_{i=1}^N a_i} \prod_{i=1}^N (X_i - a_i) 
     \]
     modulo $I(Q)$. It is easy to see that for each subset
     $I$ of $\ul{N}$, $h$ contains the monomial
     $\prod_{i \in I} X_i$. We define the \emph{support} of a monomial
     $X_1^{e_1} \cdots X_N^{e_N}$ as
     $\{ i \in \ul{N} \mid e_i \neq 0 \}$.
     The remainder of $f$ modulo $G$ is $h$, and it can
     be computed as the last element $r_n$ of a finite sequence
     $(r_k)_{k \in \ul{n}}$ of polynomials such that $r_1 = f$
     and for each $k \in \ul{n-1}$, there exists a $j \in \ul{N}$, 
     a monomial $m$ and a $c \in \ob{K}\setminus\{0\}$ such that 
     \[
     r_{k+1} = r_k - cm(X_1, \ldots, X_N) (X_j^2 - a_j X_j)
     \]
     and the monomial $m(X_1, \ldots, X_N) X_j^2$ appears in $r_{k}$ with coefficient $c$.
     (This holds because $G$ is Gr\"obner basis since all
     leading terms are coprime \cite[p.89, Exercise~11]{CLO:IVAA4}, and the remainder modulo
     a Gr\"obner basis is unique \cite[p.83, Proposition~1]{CLO:IVAA4}.)
     We will now show that for each subset $I$ of $\ul{N}$ such
     that $r_{k+1}$ contains a monomial with support $I$,
     then so does $r_k$.
     We have
     \[
       r_k = r_{k+1} + cm(X_1, \ldots, X_N) X_j^2
       - ca_j m (X_1, \ldots, X_N) X_j.
     \]  
     Let $J$ be the support of $m$ and let $J_1 := J \cup \{j\}$.
     If $I \neq J_1$, then since the support
     of both $m(X_1, \ldots, X_N) X_j^2$
     and $m (X_1, \ldots, X_N) X_j$ is $J_1$, and hence not
     $I$,
     $r_k$ contains the same monomial
     with support $I$ that is contained in $r_{k+1}$.
     If $I = J_1$ then 
     $m(X_1, \ldots, X_N) X_j^2$ is a monomial in $r_{k}$ with support $I$.
     
     From this we see that for each $I \subseteq \ul{N}$,
     $f$ contains a monomial with support $I$, and must therefore
     contain at least $2^N$ monomials.
   \end{proof}    

Given a point $\vb{a} \in (\ob{K}\setminus \{0\})^N$,
we can therefore find a nonzero of $f$ in the proximity of $\vb{a}$,
provided that $f (0,\ldots,0) \neq 0$.
\begin{cor} \label{cor:close0}
  Let $\ob{K}$ be a field, let
  $N \in \N$, let $f \in \ob{K} [X_1, \ldots, X_N] \setminus \{0\}$,
  let
  $\vb{a} = (a_1, \ldots, a_N) \in (\ob{K} \setminus \{0\})^N$,
  and let $Q := \bigtimes_{i=1}^N \{a_i, 0\}$.
  Suppose that $f(0,\ldots, 0) \neq 0$. Then there is
  $\vb{b} \in Q$ with $f(\vb{b}) \neq 0$ and
  $d_H (\vb{a}, \vb{b}) \le \log_2 (M (f))$.
\end{cor}
\begin{proof}
  As in the proof of Theorem~\ref{thm:nonzeroClose},
  we choose $\vb{b}$ to be a nonzero of $f$ with
  minimal distance to $\vb{a}$. Setting
  $k := d_H (\vb{a}, \vb{b})$, we obtain from
  Theorem~\ref{thm:coeffs2} that $M (f) \ge 2^k$ and thus
  $k \le \log_2 (M (f))$.
\end{proof}
We note that in the setting of Corollary~\ref{cor:close0},
$d_H (\vb{a}, \vb{b}) \le \log_2 (M (f))$ means 
that $\vb{b}$ has at most $\log_2 (M (f))$ zero entries.
For $J \subseteq \ul{N}$ and $f = \prod_{j \in J} (X_j - a_j)$,
the bound is attained: $M(f) =  2^{|J|}$,
and every nonzero $\vb{b}$ of $f$
in $\bigtimes_{i=1}^N \{a_i, 0\}$ must satisfy $b_j = 0$ for
all $j \in J$, and hence $\vb{b}$ contains at least
$|J| = \log_2 (M (f))$ zero entries.

\section{Applications to solving polynomial systems} \label{sec:solve}
In the field $\ob{F}_q$, the  solutions of
$f_1 = \cdots = f_r = 0$ are the nonzeros of
$\prod_{i=1}^r (1 - f_i^{q-1})$. Hence Theorem~\ref{thm:nonzeroClose}
also gives information on the solutions of polynomial systems.
\begin{thm} \label{thm:system}
  Let $\ob{F}$ be a finite field with $q>2$ elements,
  let $r, N \in \N$, let
  $f_1, \ldots, f_r \in \ob{F}[X_1, \ldots, X_N] \setminus \{0\}$,
  let $Q$ be a rectangular subset
  of $(\ob{F} \setminus \{0\})^N$, and let $t := \tfrac{q-1}{q-2}$.
  Let
  \[
    V := \{ \vb{x} \in Q \mid f_1 (\vb{x}) =
  \cdots = f_r (\vb{x}) = 0 \}.
  \]
  If $V \neq \emptyset$, then 
  for every $\vb{a} \in Q$, there is $\vb{b} \in V$
  with \[d_H (\vb{a}, \vb{b}) \le 
  \tfrac{1}{\log_2 (t)} \big(r + (q-1) \sum_{i=1}^r \log_2 (M(f_i))\big).
  \]
\end{thm}
\begin{proof}
  Let $g := \prod_{i=1}^r (1 - f_i^{q-1})$.
  Then $V = \{ \vb{x} \in Q \mid g (\vb{x}) \neq 0 \}$.
  Let $\vb{a} \in Q$.
  By Theorem~\ref{thm:nonzeroClose}, there is $\vb{b} \in V$
  with $d_H (\vb{a}, \vb{b}) \le \log_t (M(g))$.
  We have
  \[
  M(g) \le \prod_{i=1}^r (1 + M(f_i)^{q-1}),
  \]
  and therefore
  \[
      \begin{split} 
        \log_t (M(g)) & \le \sum_{i=1}^r \tfrac{\log_2 (1 + M(f_i)^{q-1})}{\log_2 (t)} \\
                   & = \tfrac{1}{\log_2 (t)} \sum_{i=1}^r \log_2 (1 + M(f_i)^{q-1}) \\
                   & \le \tfrac{1}{\log_2 (t)} \sum_{i=1}^r (1 + \log_2 (M(f_i)^{q-1})) \\
            & = \tfrac{1}{\log_2 (t)} (r + (q-1) \sum_{i=1}^r \log_2 (M(f_i))).
      \end{split}
      \]
     \end{proof}
Now we fix a finite field $\ob{F}$,
$r \in \N$ and a subset $S$ of $\ob{F} \setminus \{0\}$
and consider the problem to determine on input
$f_1, \ldots, f_r \in \ob{F} [X_1, \ldots, X_N]$
whether $f_1 = \cdots = f_N = 0$ has a solution in $S^N$.
We 
measure the size of the input polynomials in such a way
  that the size $n$ of $f_1, \ldots, f_r$ is at least $\max(N,M)$, where $N$ is the
  number of variables and $M$ is their total number of monomials.
  Then adapting the idea of~Corollary~\ref{cor:testset},
  we obtain an algorithm of time complexity in $O(n^{c \log(n)})$ (with
  $c > 0$) to
  solve this question.

  Theorem~\ref{thm:system} has a consequence that reminds of a Theorem by
  Chevalley \cite{Ch:DDHD} if we measure the complexity of a polynomial rather
  by
  the number of monomials it contains than by its degree.
 
  \begin{cor} \label{cor:CW}
  Let $\ob{F}$ be a finite field with $q>2$ elements,
  let $r, N \in \N$, let
  $f_1, \ldots, f_r \in \ob{F}[X_1, \ldots, X_N] \setminus \{0\}$,
  let $Q = \prod_{i = 1}^N A_i$ be a rectangular subset
  of $(\ob{F} \setminus \{0\})^N$ with $|A_i| > 1$ for all
  $i \in \ul{N}$, and let $t := \tfrac{q-1}{q-2}$.
  Let
  \[
    V := \{ \vb{x} \in Q \mid f_1 (\vb{x}) =
  \cdots = f_r (\vb{x}) = 0 \}.
  \]
  If 
      \( N > 
  \tfrac{1}{\log_2 (t)} \big(r + (q-1) \sum_{i=1}^r \log_2 (M(f_i))\big),
  \)
  then $V$ is not a singleton.
  \end{cor}
   \begin{proof}
      We assume that $V$ is nonempty and  $\vb{c} \in V$.
      Let $\vb{a} \in Q$ be such that $d_H (\vb{a}, \vb{c}) = N$;
      such an $\vb{a}$ exists because each set $A_i$ contains
      at least two elements, which allows to pick
      $\vb{a} \in Q$ that differs from $\vb{c}$ in all components.
      Now Theorem~\ref{thm:system} yields $\vb{b} \in V$
      with $d_H (\vb{a}, \vb{b}) \le
      \tfrac{1}{\log_2 (t)} (r + (q-1) \sum_{i=1}^r \log_2 (M(f_i)))$,
      and thus by assumption, $d_H (\vb{a}, \vb{b}) < N$,
      which implies $\vb{b} \neq \vb{c}$.
      Hence $V$ is not a singleton.
  \end{proof}
 We note that the Theorem by Schauz and Brink
  \cite[Theorem~1]{Br:CTWR},
  \cite[Corollary~3.5]{Sc:ASPD} (cf. \cite[Theorem~13.1]{AM:CWTR})
  has the same conclusion -- the solution set in a rectangular domain
  is not a singleton -- under different hypotheses.
  There are situations in which Corollary~\ref{cor:CW} can be applied,
  but the assumptions of the Schauz-Brink-Theorem are not satisfied.
  Let us give one such example: by a \emph{binomial}, we understand
  a polynomial containing exactly $2$ monomials.
  On $\ob{F}_3$, let $Q := (\ob{F}_3 \setminus \{0\})^N$,
  and suppose that we are given a set of $r \ge 2$ binomials
  $f_1, \ldots, f_r$ in $N > 3r$ variables such that
  each of the $f_i$'s  has degree at least $\frac{N}{2}$. 
  The assumption under which the Schauz-Brink-Theorem
  claims that $f_1 = \cdots = f_r = 0$ does not have exactly one solution
  is $N > \sum_{i=1}^r \deg(f_j)$, which is not satisfied in the
  given situation, whereas
  the assumption of Corollary~\ref{cor:CW} is $N > 3r$, which
  we assumed to be satisfied.

  In the case that we know that $(0,\ldots, 0)$ is one solution of
  our system, Corollary~\ref{cor:close0} allows
  to obtain smaller spheres containing solutions.
  \begin{thm}
  Let $\ob{F}$ be a finite field with $q$ elements,
  let $r, N \in \N$, and let
  $f_1, \ldots, f_r \in \ob{F}[X_1, \ldots, X_N] \setminus \{0\}$.
  Let $Q = \bigtimes_{i=1}^N A_i$ be a rectangular subset of $\ob{F}^N$ with
  $(0, \ldots, 0)  \in Q$,
  and let 
  \[
    V := \{ \vb{x} \in Q \mid f_1 (\vb{x}) =
  \cdots = f_r (\vb{x}) = 0 \}.
  \]
  If $(0,\ldots, 0) \in V$, then 
  for every $\vb{a} \in \bigtimes_{i=1}^N (A_i \setminus \{0\})$, there is $\vb{b} \in V \cap
   (\bigtimes_{i=1}^N \{0 , a_i\})$
  with \[d_H (\vb{a}, \vb{b}) \le 
           r + (q-1) \sum_{i=1}^r \log_2 (M(f_i)).
  \]
\end{thm}
  \begin{proof}
    We use Corollary~\ref{cor:close0} for the
    polynomial $g = \prod_{i=1}^r (1-f_i^{q-1})$
    and thereby obtain
     $\vb{b} \in V \cap (\bigtimes_{i=1}^N \{0 , a_i\})$
    with $d_H (\vb{a}, \vb{b}) \le \log_2 (M (g))$.
       Now $\log_2 (M (g)) \le
            \sum_{i=1}^r \log_2 (1 + M(f_i)^{q-1})
           \le \sum_{i=1}^r (1 + \log_2 (M(f_i)^{q-1}))
       = r + (q-1) \sum_{i=1}^r \log_2 (M(f_i))$.
  \end{proof}
  \section*{Acknowledgements}
  The authors thank Arne Winterhof for drawing their attention
  to the reference~\cite{KW:OTIO}.
  \bibliography{coeffs36}
    \end{document}